\documentclass[10pt]{amsart}
\usepackage{amsmath,amsthm}
\usepackage[colorlinks=true,linkcolor=blue,urlcolor=blue,citecolor=blue]{hyperref}
\usepackage[all]{xy}
\usepackage[usenames,dvipsnames,svgnames,table]{xcolor}
\usepackage[latin1]{inputenc}

\newtheorem{theorem}{Theorem}[section]
\newtheorem{lemma}[theorem]{Lemma}

\newtheorem{definition}[theorem]{Definition}

\theoremstyle{remark}
\newtheorem{remark}[theorem]{Remark}

\numberwithin{equation}{section}


\title
{{Locally invertible $\sigma$--harmonic mappings}}

\author[G. Alessandrini]
{Giovanni Alessandrini}
\address[G. Alessandrini]{Dipartimento di Matematica e Geoscienze, Università  di Trieste, Via Valerio 12/b, 34100 Trieste, Italia} 
\email[G. Alessandrini]{alessang@units.it}
\author[V. Nesi]
{Vincenzo Nesi}
\address[V.  Nesi]{Dipartimento di Matematica ``G. Castelnuovo", Sapienza, Università di Roma,
Piazzale A. Moro 2, 00185 Roma, Italy}
\email[V. Nesi]{nesi@mat.uniroma1.it}

\begin{document}

\begin{abstract}
We extend a classical theorem by H. Lewy to planar $\sigma$--harmonic mappings, that is mappings $U$ whose components $u^1$ and $u^2$ solve a divergence structure elliptic equation ${\rm div} (\sigma \nabla u^i)=0$ , for $i=1,2$. A similar result is established for pairs of solutions of certain second order  non--divergence equations.
\end{abstract}
\maketitle
\begin{flushright}
\textit{\small Dedicato, in occasione del suo ottantacinquesimo compleanno, a Gianfausto,\\
mentore di V.N. e ``motore di ricerca'' precedente all'invenzione del web.
}\end{flushright}

\section{Introduction}
The fundamental properties of the zeroes of holomorphic functions tell us that, if $f$ is a non--constant holomorphic function near $0$ and $f^\prime (0)=0$, then
\[f(z)-f(0) = \chi(z)^m \ , \text{ in a neighborhood of } 0\ ,\]
where $m\ge2$ is an integer and $\chi$ is a conformal map which fixes the origin. Hence, $f$ is locally invertible if and only if $f^\prime\neq 0$.

Consider an open set $\Omega\subseteq \mathbb R^2$ and a mapping $U=(u^1,u^2):\Omega\to \mathbb R^2$ such that $\Delta u^i=0, i=1,2$. Such a mapping is called a (planar) harmonic mapping. If $u^2$ is {  a} harmonic conjugate of $u^1$, and we use the customary convention to identify points $x=(x_1, x_2)\in \mathbb R^2$ with complex numbers $z=x_1+ix_2$, then the mapping $U=u^1+i u^2$ is holomorphic and the condition $U^{\prime}(z)\neq 0$ reads $\det DU > 0$. 
A classical, remarkable, extension of the property that  holomorphic functions are locally injective if and only if $\det DU>0$,  was proven by Hans Lewy \cite{Lewy}.  If $U: \Omega \to D$ is a harmonic homeomorphism, then its Jacobian matrix $DU$ is nonsingular, see also Duren \cite{duren}. 

Notice that the limitation to two dimensions is natural in this context in view of the explicit counterexample by Wood
\cite{wood}  of a three dimensional harmonic homeomorphism which is not a diffeomorphism.

A partial generalization of Lewy's result was obtained by the authors for invertible, sense preserving, $\sigma$--harmonic mappings $U$, that is maps whose components $u^1$ and $u^2$ solve a divergence structure elliptic equation 
\begin{equation}\label{dvrg}{\rm div} (\sigma \nabla u^i)=0\ , i=1,2\ , 
\end{equation}
when $\sigma$ is a strictly positive definite matrix with $L^{\infty}$ entries. In such a case, since solutions are differentiable in the weak sense only, the best possible result is that 
\[ {\rm det}DU > 0 \text{ almost everywhere} , \]
see \cite{ANARMA} and also \cite{ANFINNICO},  where  non--symmetric coefficient matrices  are taken into account.

It is worth recalling that $U$ may fail to be a diffeomorphism when $\sigma$ is discontinuous just at a single point. The following example is taken from Meyers \cite{meyers}.

Let $\alpha>0$  and let us consider
\begin{equation*}
\sigma(x)=
\left(
\begin{array}{cc}
\frac{\alpha^{-1}x_1^2 +\alpha x_2^2}{x_1^2+x_2^2}&\frac{(\alpha^{-1}-\alpha)x_1 x_2}{x_1^2+x_2^2}\\
\\
\frac{(\alpha^{-1}-\alpha)x_1 x_2}{x_1^2+x_2^2}&\frac{\alpha x_1^2 +\alpha^{-1} x_2^2}{x_1^2+x_2^2}
\end{array}
\right)\,,
\end{equation*}
 so that $\sigma$ has eigenvalues $\alpha$ and $\alpha^{-1}$ and it is uniformly elliptic. However
$\sigma$ is discontinuous at $0$, and only at $0$, when $\alpha\neq 1$.
Set
\begin{equation*}
\begin{array}{l}
u^1(x)=|x|^{\alpha -1}x_1\,,\\
u^2(x)=|x|^{\alpha -1}x_2\,.
\end{array}
\end{equation*}
A direct calculation shows that $U=(u^1,u^2)$ is $\sigma$--harmonic and injective.
We compute 
\begin{equation*}
\det DU =\alpha |x|^{2(\alpha-1)}\,.
\end{equation*}
Therefore $\det DU$ vanishes at $0$ when $\alpha>1$,  whereas it diverges when $\alpha\in (0,1)$.

Let us also recall an example, provided by Hartman and Wintner \cite[Theorem ($i^*$)]{HW}, of a coefficient matrix $\sigma$, uniformly elliptic and with \emph{continuous} entries, such that the only weak solutions to $ {\rm div}  (\sigma \nabla u) = 0$ which have continuous first derivatives are the constant ones. As a consequence, for such a $\sigma$, no  homeomorphic $\sigma$--harmonic mapping $U$ can be a diffeomorphism.

The aim of the present note is to show that  a homemomorphic $\sigma$--harmonic mapping $U$ satisfies
\[ {\rm det}DU\neq 0 \text{ everywhere} , \]
when the entries of $\sigma$ are H\"older continuous. We recall that, in this case, the local H\"older continuity of $D U$ is well--known.

Let us now state our main result. In what follows we let $\Omega \subset \mathbb R^2$ be a  simply connected open set, and we denote by  $\sigma=\sigma(x)$  a possibly non--symmetric matrix  having measurable entries and satisfying the ellipticity conditions
\begin{equation}\label{ell}
\begin{array}{ccrllll}
\hbox{$\hskip0,4cm \sigma(x) \xi\cdot \xi \geq  K^{-1} |\xi|^2$, for every $\xi \in \mathbb R^2\ , x \in \Omega$\,,}\\

\hbox{$\sigma^{-1}(x) \xi \cdot \xi \geq  K^{-1} |\xi|^2$, for every $\xi \in \mathbb R^2 \ , x \in \Omega$\,,}

\end{array}
\end{equation}
for a given constant $K\ge 1$.

\begin{theorem}\label{main}
Assume {  \eqref{ell} and } that the entries of $\sigma$ satisfy $\sigma_{ij}\in C^{\alpha}_{loc}(\Omega)$ for some $\alpha \in (0,1)$ and for every $i,j=1,2$.
Let $U=(u^1,u^2) \in W^{1,2}_{loc}(\Omega, \mathbb R^2)$  be such that
\begin{equation}\label{basiceq}
{\rm div} (\sigma \nabla u^i)=0 \ , i=1,2,
\end{equation}
weakly in $\Omega$. If $U$ is  locally a homeomorphism, then it is, locally, a  diffeomorphism,  {   that is}
\begin{equation}\label{det}
{\rm det} DU\neq 0 \hbox{ for every } x \in \Omega\ .
\end{equation}
\end{theorem}

The main object of investigation here is merely of local character. We should also mention, however,  the relevance of the global issues regarding finding suitable boundary data guaranteeing that a harmonic (or $\sigma$--harmonic) map $U$ is a homeomorphism, or a diffeomeorphism, in the large. Starting with the classical Rad\'{o}--Kneser--Choquet theorem \cite{duren}, let us mention the contributions by Bauman and Phillips \cite{BP}, Bauman, Marini and Nesi \cite{bmn}, and also by the present authors \cite{ANARMA,ANPISA}. For an holomorphic function the results are classical. Let $f$ be holomorphic in the open set $\Omega$, and let $\omega$ be a Jordan domain with boundary $\gamma$ such that $\overline{\omega}=\omega\cup \gamma \subset \Omega$.  When $f$ is one-to-one on $\gamma$, $f(\gamma)$ is a Jordan curve $\Gamma$ which is the boundary of the open set $f(\omega)$. Then $f$ maps the bounded set $\omega\cup \gamma$ onto the set $f(\omega)\cup \Gamma$ in  a one to one way. Note that the condition that $f$ is a one to one map from $\gamma$ to $\Gamma$ is necessary. The classical statement is that the latter condition is also sufficient.
See for instance \cite[Theorem 4.5]{Markushevich}, and the discussion and far--reaching extensions by Meisters and Olech \cite{meistol}.
 Planning to return on such questions in forthcoming research, for the purpose of this note, we restrict  the attention to the purely local issue.

In the next Section \ref{pre} we prepare the proof of Theorem \ref{main} with some preliminary considerations and two Lemmas, { and eventually} we conclude the proof. The final Section \ref{nondiv} contains a variation on the theme of Theorem \ref{main}, in which we treat the case when the equation \eqref{dvrg} is replaced by an equation in non--divergence form with $L^{\infty}$ coefficients. In fact, it is well--known, Bers and Nirenberg \cite{bersni}, Talenti \cite{ta}, that, in two dimensions, $W^{2,2}_{loc}$ solutions of non--divergence elliptic equations have H\"older continuous first derivatives, even if the coefficients in the principal part are discontinuous. Hence it makes sense to enquire if Theorem \ref{main} can be extended to this case. The affirmative answer is contained in Theorem \ref{var}.

\section{{ Preliminary Lemmas and the proof of Theorem \ref{main}}}\label{pre}

In all what follows we shall assume the ellipticity conditions \eqref{ell} to be satisfied.  Recall that we use the convention to identify points $x=(x_1, x_2)\in \mathbb R^2$ with complex numbers $z=x_1+ix_2$.
Let us recall some known facts on solutions of elliptic equations in two variables. 
Let $u \in W^{1,2}_{loc}(\Omega)$ be a weak, real valued, solution to
\[{\rm div} (\sigma \nabla u)=0 \hbox{ in } \Omega \ , \]
then there exists $v \in W^{1,2}_{loc}(\Omega)$ (called the \emph{stream function} of $u$) such that
\begin{equation}\label{CR}
\nabla v = J \sigma \nabla u \ ,
\end{equation}
where the matrix $J$ represents the counterclockwise $90^{\circ}$ rotation
\begin{equation}\label{J}
 J=\left(
\begin{array}{ccc}
0&-1\\
1&0
\end{array}
\right)\,,
\end{equation}
see, for instance, \cite{AMsiam}. The system \eqref{CR} can be recast as a Beltrami type equation. In fact, if we set \begin{equation}\label{f}
f=u+iv \ , 
\end{equation}
then \eqref{CR} can be rewritten as
\begin{equation}\label{B}
\begin{array}{ll}
f_{\bar{z}}=\mu f_z +\nu \overline{f_z}\ & \hbox{in $\Omega$}\ ,
\end{array}
\end{equation}
where, the so called complex dilatations $\mu , \nu$ are given by
\begin{equation}\label{SNU}
\begin{array}{llll}
\mu=\frac{\sigma_{22}-\sigma_{11}-i(\sigma_{12}+\sigma_{21})}{1+{\rm
Tr\,}\sigma +\det \sigma}& \ ,&\nu =\frac{1-\det \sigma
+i(\sigma_{12}-\sigma_{21})}{1+{\rm Tr\,}\sigma +\det \sigma}\ ,
\end{array}
\end{equation}
and satisfy the following
ellipticity condition
\begin{equation}\label{ellQC}
|\mu|+|\nu|\leq k< 1 \,,
\end{equation}
where the constant $k$ only depends on $K$, see \cite[Proposition 1.8]{ANFINNICO} and, for any $2\times 2$ matrix $A$, the trace of $A$ is denoted by ${\rm Tr\,} A$.

It is a classical well--known fact, Bers and Nirenberg \cite{bersni}, Bojarski \cite{BO}, that a $W^{1,2}_{loc}$ solution to \eqref{B} fulfills the so--called Stoilow representation
\begin{equation}\label{St}
f= F\circ \chi \ ,
\end{equation}
where $F$ is holomorphic and $\chi$ is a quasiconformal homeomorphism. As an immediate consequence, $u$ can be represented as
\[ u= h \circ \chi \ , \]
where $h$ is harmonic. Thus, up to a quasiconformal mapping, the structure of the level lines of $u$ is the same as the one of a harmonic function. In this respect, in \cite{AMsiam} the concept of \emph{geometrical critical point} was introduced as follows: $z_0\in \Omega$ is a geometrical critical point for $u$ if and only if $\chi(z_0)$ is a critical point for $h$. In \cite[Theorems 2.7, 2.8]{AMsiam} it was also introduced a calculus of geometric critical points in terms of the oscillatory character of prescribed (Dirichlet or Neumann) boundary data for $u$. We shall apply such a calculus in a very specific case, to this purpose we recall a terminology first introduced in \cite{ln}.

\begin{definition}\label{umod}
Let $G$ be a Jordan domain bounded by the Jordan curve $\Gamma$. A non--constant continuous function $g$ on $\Gamma$ is said to be \emph{unimodal} if $\Gamma$ can be split into two simple arcs $\Gamma_1, \Gamma_2$, which inherit the orientation of $\Gamma$, such that $g$ is nondecreasing on $\Gamma_1$ and nonincreasing on $\Gamma_2$.
\end{definition}
The argument to prove the following Lemma can be traced back to Kneser \cite{k}, in his proof of the celebrated Rad\`o--Kneser--Choquet theorem, see for instance Duren \cite{duren}.
\begin{lemma}\label{nogcp}
Let $G$ be as in Definition \ref{umod}, given $g$ continuous on $\partial G = \Gamma$, consider the weak solution $u\in W^{1,2}_{loc}(G)\cap C(\overline{G})$ to the Dirichlet problem
\begin{equation}\label{dirich}
\left\{
\begin{array}{lll}
{\rm div} (\sigma \nabla u)=0&{\rm in}& G\,,\\
u=g&{\rm on}&\partial G\,.
\end{array}
\right.
\end{equation}
If $g$ is unimodal, then $u$ has no geometrical critical points and the mapping $f$ given by \eqref{f} is a quasiconformal homeomorphism.
\end{lemma}
{  
\begin{remark}\label{regbdry} It is a classical matter that a unique solution to \eqref{dirich} exists, indeed we may recall the  theory by Littman, Stampacchia and Weinberger \cite{LSW} and the fact that the boundary  points of a Jordan domain are regular for the classical Dirichlet problem, see for instance \cite[Ch. XII]{kell}.  However, we emphasize that this is not the central issue here, since this Lemma will be applied to restrictions of solutions on larger domains, which will be automatically continuous up to the boundary.
\end{remark}
}
\begin{proof}
The absence of geometrical critical points was proven in \cite[Theorem 2.7]{AMsiam}, { to which we refer for details. We should note that  in \cite{AMsiam}  $\sigma$ is assumed symmetric, but the proof applies with no changes also in the non--symmetric case. In fact, up to the change of coordinates $\chi$, the whole matter reduces to analyze a harmonic function $h$ whose Dirichlet data is unimodal (Kneser \cite{k} first proved that under such assumptions $\nabla h$ never vanishes).}

Then the representation formula \eqref{St} can be rewritten  
\begin{equation}\label{St2}
u= h\circ \chi \ , \, v= k\circ \chi \ ,
\end{equation}
with $\nabla h \neq 0$ and $k$ being a harmonic conjugate to $h$. Moreover, the unimodality of $g$ also implies that the level lines of $h$ are simple arcs and {  then, using the Cauchy--Riemann equations, we deduce that} $k$ is strictly monotone along them. Hence $F=h+ik$ is an injective holomorphic map and consequently $f=F\circ \chi$ is a quasiconformal homeomorphism.
\end{proof}
A variant of the previous Lemma can be formulated as follows.
\begin{lemma}\label{lemma2}
In addition to the hypotheses of Lemma \ref{nogcp}, let us assume $\sigma_{ij}\in C^{\alpha}_{loc}(G)$ for some $\alpha \in (0,1)$ and for every $i,j=1,2$. Then we have
\[ |\nabla u | >0 \hbox{ everywhere in } G\ . \]
\end{lemma}
\begin{proof}
In view of \eqref{SNU}  and \eqref{ell}, the coefficients $\mu, \nu$ in equation \eqref{B}, turn out to be $C^{\alpha}_{loc}(G,\mathbb C)$. { In view of Lemma \ref{nogcp} we may introduce} $g = f^{-1}$, the Beltrami equation for $g$ can be computed to be
\begin{equation}\label{inverseB}
g_{\overline{w}}=-\nu(g) g_{w}- \mu(g)\overline{g_{w}}\,,
\end{equation}
see for instance \cite{ANves}. Also in this equation, the coefficients belong to $C^{\beta}_{loc}$ for some $\beta \in (0,1)$. Classical interior regularity theory tells us that $g \in C^{1,\beta}_{loc}$. Hence, for any compact subset $Q\subset G$  there exists $C>0$ such that
\begin{equation}\label{upperbound}
\begin{array}{lll}
|g_w|^2-|g_{\overline{w}}|^2\leq C^2& \hbox{in}&f(Q)\,,
\end{array}
\end{equation}
which can be rewritten as 
\begin{equation*}
\begin{array}{lll}
|f_z|^2-|f_{\overline{z}}|^2\geq C^{-2}& \hbox{in}&Q\,,
\end{array}
\end{equation*}
which in turn implies
\begin{equation}
\begin{array}{lll}
|\nabla u|\geq C^{-1}>0& \hbox{in}&Q\,.
\end{array}
\end{equation}
\end{proof}

{  We  are now in position to complete the proof of our main result.}
\begin{proof}[Proof of Theorem \ref{main}]
 Up to replacing $\Omega$ with a smaller open subset, there is no loss of generality in assuming that $U$ is one--to--one in all of  $\Omega$. It suffices to prove that for all $\xi \in \mathbb R^2$, $|\xi|=1$, the function $u = U\cdot \xi$ satisfies 
\[ |\nabla u | >0 \hbox{ everywhere in } \Omega . \]
 By linearity, $u$ solves
\[{\rm div} (\sigma \nabla u)=0 \]
in $\Omega$.

Let us fix $z_0\in\Omega$ and set $ w_0=U(z_0)$. Let $r>0$ be such that $\overline{B_r(w_0)}\subset U(\Omega)$. Let $G=U^{-1}(B_r(w_0) )$. $G$ is a Jordan domain and $U(\partial G)= \partial B_r(w_0)$ is a circle, hence the boundary of a convex domain. As a consequence, $g = u|_{\partial G}= U\cdot \xi|_{\partial G}$ is unimodal. By Lemma \ref{lemma2}, the thesis follows.
\end{proof}
\begin{remark}
It may be curious to notice that, in the above proof, use is made, on a local basis, of an argument based on the convexity of a domain in the target coordinates, which is crucial in the already mentioned Rad\'{o}--Kneser--Choquet theorem and its known variants.
\end{remark}
\section{The non--divergence case}\label{nondiv}
\begin{theorem}\label{var}
Let $U=(u^1,u^2) \in W^{2,2}_{loc}(\Omega, \mathbb R^2)$  be such that
\begin{equation}\label{nondiveq}
{\rm Tr} (\sigma D^2 u^i)+ b\cdot \nabla u^i=0 \ , i=1,2,
\end{equation}
almost everywhere in $\Omega$, where $\sigma$  fulfills \eqref{ell} and $b\in L^{\infty}(\Omega, \mathbb R^2)$. 

If $U$ is  locally a homeomorphism, then it is,  locally,  a diffeomorphism and
\begin{equation}\label{det>}
{\rm det} DU\neq 0 \hbox{ for every } x \in \Omega\ .
\end{equation}
\end{theorem}
 The following Lemma, in the style of Lemma \ref{lemma2}, shall be needed.
\begin{lemma}\label{nocp3}
Let $G$ and $g$ be as in Lemma \ref{nogcp}, assume that there exists  $u\in W^{2,2}_{loc}(G)\cap C(\overline{G})$  which solves the Dirichlet problem
\begin{equation}\label{dirichnondiv}
\left\{
\begin{array}{lll}
{\rm tr} (\sigma D^2 u)+ b\cdot \nabla u=0 &{\rm in}& G\,,\\
u=g&{\rm on}&\partial G\,.
\end{array}
\right.
\end{equation}
If $g$ is unimodal then $u$ has no critical points.\end{lemma}

{  
\begin{remark} Notice that the existence of a solution to \eqref{dirichnondiv} is taken as an assumption. In fact, as already noted in Remark \ref{regbdry}, this Lemma (similarly to Lemma \ref{nogcp}) will be applied to restrictions of solutions on larger domains.
\end{remark}
}

\begin{proof}
Also for equations in non--divergence form a reduction to a Beltrami type equation is possible, \cite{bersni}, in this context the ad--hoc unknown is the complex derivative $f=\partial_z u$.   We omit the well--known calculation. It suffices to say that also in this case a calculus on the number of critical points in terms of the oscillation character of the Dirichlet data has been developed, \cite[Theorem 4.1]{APisa}, in particular, if $g$ is unimodal then $\nabla u$ never vanishes.
\end{proof}

\begin{proof}[Proof of Theorem \ref{var}]
We may follow the line of the proof of Theorem \ref{main}, just by invoking Lemma \ref{nocp3} in place of Lemma \ref{lemma2}.
\end{proof}
\begin{remark}
As a consequence of Theorem \ref{var}, we observe that a further variant of Theorem \ref{main} could be obtained if the H\"older continuity assumption on the entries of $\sigma$ was replaced by the assumption that, in the weak sense,
\[ {\rm div} \sigma = (\partial_{x_1}{\sigma_{11}}+\partial_{x_2}{\sigma_{21}},\partial_{x_1}{\sigma_{12}}+\partial_{x_2}{\sigma_{22}}) \in L^{\infty}(\Omega, \mathbb R^2)\,. \]
In fact,  under this assumption, the divergence structure equation \eqref{dvrg} can be transformed, up to a customary regularization procedure, see for instance \cite{ANJAM}, into the non--divergence form appearing in Theorem \ref{var}, {  
and $W^{1,2}_{loc}$ solutions are indeed $W^{2,2}_{loc}$--regular.
}
\end{remark}
{\bf Acknowledgements} G.A. was supported by 
Universit\`a degli Studi di Trieste FRA 2016, 
V.N. was supported by Fondi di Ateneo Sapienza ``Metodi di Analisi Reale e Armonica per problemi stazionari ed evolutivi''.

\end{document}